\documentclass{amsart} 
\title{Moduli of Singular Curves}
\author{Jack Hall}

        \usepackage{latexsym}
        \usepackage{amssymb}
        \usepackage{amsmath}
        \usepackage{amsfonts}
        \usepackage{amsthm}
        \usepackage[mathcal]{eucal}
        \usepackage{fancyhdr}
        \usepackage[hypertexnames=false]{hyperref}
        \usepackage[pdftex]{graphicx}
        \usepackage[all,knot,poly,2cell]{xy}

\usepackage{palatino}
\usepackage{euler}

        \newcommand{\tensor}{\otimes}

        \newcommand{\C}{\mathbb{C}}
        \newcommand{\Z}{\mathbb{Z}}

        \newcommand{\Coll}{\mathscr{C}}

        \renewcommand{\bar}{\overline}

        \newcommand{\Orb}{\mathscr{O}}
        \DeclareMathOperator{\Hom}{Hom}
        \DeclareMathOperator{\Ext}{Ext}

        \DeclareMathOperator{\Obj}{Obj}

        \newcommand{\ideal}{\triangleleft}      
        \DeclareMathOperator{\Pic}{Pic}
       \DeclareMathOperator{\spec}{Spec}
       
        \DeclareMathOperator{\Isom}{Isom}
        \DeclareMathOperator{\Def}{Def}
 \DeclareMathOperator{\Exal}{Exal}    

\renewcommand{\Pr}{\mathbb{P}\,}

        \theoremstyle{plain}
        \newtheorem{thm}{Theorem}[section]
        \newtheorem{cor}[thm]{Corollary}
        \newtheorem{lem}[thm]{Lemma}
        \newtheorem{prop}[thm]{Proposition}

        \theoremstyle{definition}

        \theoremstyle{remark}

\newcommand{\SCHABS}{\mathbf{Sch}}

\newcommand{\SCH}[1]{\SCHABS/{#1}}

\newcommand{\SETS}{\mathbf{Sets}}

\newcommand{\BETSITE}[1]{({#1})_{\tiny{\mbox{\'Et}}}}

\newcommand{\curv}{\mathcal{U}}
\newcommand{\bddcurv}{\mathcal{W}}
\newcommand{\HILB}[1]{\underline{\mathrm{Hilb}}_{#1}}

\newcommand{\ART}[1]{\mathbf{Art}_{#1}}
\newcommand{\OPP}{{\mathrm{opp}}}

\newcommand{\red}[1]{{#1}_{\mathrm{red}}}

\thanks{We would like to sincerely thank Jarod Alper, David Rydh,
  David Smyth, Ravi Vakil, and Fred Van Der Wyck for their comments 
  and suggestions.}
\begin{document}
\maketitle
\begin{abstract}
  The purpose of this note is to prove that there
  is an algebraic stack $\mathcal{U}$ parameterizing all
  curves. The curves that appear in the algebraic stack $\mathcal{U}$
  are allowed to be arbitrarily singular, non-reduced, disconnected,
  and reducible. We also
  prove the boundedness of the open substack of $\mathcal{U}$
  parameterizing geometrically connected curves with fixed
  arithmetic genus $g$ and $\leq e$ irreducible components. This is an  
  expanded version of \cite[Appendix
  B]{smyth-2009}.
\end{abstract}
\section{Introduction}\label{sec:intro}
Fix a scheme $S$. For an $S$-scheme $T$, a $T$-curve is defined to be a
proper, flat, and finitely presented morphism of algebraic spaces $\pi
: C \to T$, where the geometric fibers have dimension $1$. By \cite[Theorem V.4.9]{MR0302647},
\cite[Exercise III.5.8]{MR0463157} and \cite[\textbf{IV}, 9.1.5]{EGA},
the geometric fibers of a $T$-curve are projective.  Let $\SCH{S}$
denote the category of $S$-schemes and define $\curv_S$ to be the
\'etale stack over $\SCH{S}$, which assigns to each $S$-scheme $T$,
the groupoid of $T$-curves.   

It is tempting to restrict attention to $T$-curves $\pi : C \to T$,
where the map $\pi$ is \emph{projective}. Indeed, if $T$ is an
affine scheme, then any smooth $T$-curve $X \to T$ with geometric
fibers of genus $g \neq 1$ is a projective $T$-scheme. In the case that $g=1$, there is an
example due to M. Raynaud, which 
appears in \cite[XIII-3.1]{MR0260758}, of a family of
elliptic curves, over an affine base, which is Zariski locally
projective, but not projective. There is also an example due to
D. Fulghesu, appearing in \cite[Example 2.3]{fulghesu-2009}, of a
proper algebraic 3-fold, fibered over a projective surface, which is
a family of nodal curves of genus 0, with at most 2 nodes in each
fiber, which is \emph{not} a scheme. In particular, this family is
not Zariski locally projective. Thus when parameterizing
singular curves, the total spaces of the families are required to be
algebraic spaces. We will prove the following:
\begin{thm}\label{thm:stack_of_curv}
  $\curv_S$ is an algebraic stack, locally finitely presented over
  $S$, with quasi-compact and separated diagonal. There is an
  explicit, smooth cover of $\curv_S$ by Hilbert schemes of projective
  spaces. 
\end{thm}
We note that proofs of the algebraicity of $\curv_S$ have recently
appeared in \cite[Prop.\ 2.3]{dejong-2008} and
\cite{lundkvistthesis} using Artin's Criterion
\cite[Thm.\ 5.3]{MR0399094}. We provide a proof logically independent
of Artin's 
Criterion [\emph{loc. cit.}], by constructing an explicit
presentation by Hilbert schemes of projective spaces. Theorem 
\ref{thm:stack_of_curv} and the corollaries that follow were used by
\cite{smyth-2009} in the production of alternate compactifications of 
$\mathcal{M}_{g,n}$.  
\begin{cor}\label{cor:int_vers}
  If $C \to \spec \Bbbk$ is a projective curve, then it has a versal
  deformation space defined by equations with integral
  coefficients. 
\end{cor}
We observe that Corollary \ref{cor:int_vers} is a trivial corollary of
Theorem \ref{thm:stack_of_curv}, yet at face value it is entirely
non-obvious. For example, if you were to consider a complex curve $C
\to \spec \C$, with defining equations in some embedding into
$\Pr^N_\C$ having lots of transcendental terms, then you would
certainly not expect the deformation theory to be governed by
equations with integral coefficients. Since the versal deformation of
a rigid curve is itself, we immediately obtain the following partial
answer to a speculation of R. Vakil in \cite{MR2227692}: 
\begin{cor}\label{cor:int_vers2}
  If $C \to \spec \Bbbk$ is a rigid, projective curve, then every
  singularity type of $C$ is defined over $\Z$.
\end{cor}
The following corollaries show that from the construction of the
algebraic stack $\curv_S$, one easily obtains
fine moduli stacks of essentially every other moduli problem
associated to curves.
\begin{cor}\label{cor:stack_of_curv1}
  The stack $\curv_{S,n}$ whose objects are curves + $n$ arbitrary
  sections is algebraic, locally finitely presented over $S$, with
  quasi-compact and separated diagonal. 
\end{cor}
\begin{cor}\label{cor:stack_of_curv2}
  One may impose any number of the following extra conditions on the
  morphisms in $\Obj \Coll_S$ and still obtain an algebraic $S$-stack
  which is locally finitely presented over $S$, with quasi-compact and
  separated diagonal:
  \begin{enumerate}
  \item geometric fibers are $R_n$;
  \item geometric fibers are $S_n$;
  \item geometric fibers are lci;
  \item geometric fibers are Cohen-Macaulay;
  \item geometric fibers are reduced with $k$ connected components;
  \item geometric fibers are reduced;
  \item geometric fibers are reduced and connected;
  \item geometric fibers are reduced, connected, and have $e$ or fewer 
    irreducible components;
  \item geometric fibers are integral;
  \item geometric fibers have arithmetic genus $g$;
  \item geometric fibers have quasi-finite automorphism group;
  \item geometric fibers have no infinitesimal automorphisms;
  \item any condition on a flat family of curves specified by a
    condition on an open function (e.g. a polynomial) in the
    cohomology groups on the fibers of a finite set of complexes of
    sheaves with coherent cohomology (possibly not flat), all of which
    respect pullback along the base (for example $h^1(L_{X/T})=3$).  
  \end{enumerate}
  In particular, (11) defines the largest substack of $\curv$ with
  quasi-finite diagonal, and (12) defines the largest Deligne-Mumford
  substack. 
\end{cor}
We also prove the following boundedness result, such a result was
believed to exist, but there was no proof in the literature. 
\begin{cor}\label{cor:stack_of_curv3}
  For any fixed triple of integers $(g,n,e)$, the stack
  $\bddcurv_{S,g,n,e}$ corresponding to geometrically connected,
  reduced curves of arithmetic genus $g$, with $n$ marked points, and
  $e$ or fewer irreducible components is algebraic, \emph{finitely}
  presented over $S$, with quasi-compact and separated diagonal.
\end{cor}
Theorem \ref{thm:stack_of_curv} and its corollaries will be proved in
the subsequent sections. 
\section{\'Etale local projectivity}
We will show that a $T$-curve $C \to T$ is \'etale-locally
projective. Note that this result is an immediate consequence of Artin
approximation, but we provide an independent proof.  
An important preliminary observation is that
$\curv_S$ is a \emph{limit preserving} stack. That is, if
$\{A_j\}_{j\in J}$ is an inductive sequence of $S$-rings, we set $A
=\varinjlim_j A_j$, then the natural 
transformation: 
\[
\varinjlim_j (\curv_S)_{\spec A_j} \longrightarrow (\curv_S)_{\spec A}
\]
is an equivalence of categories. In concrete terms, it means that if
you have a $\spec A$-curve $X \to \spec A$, there is some $j \in J$
and a $\spec A_j$-curve $X_j \to \spec A_j$ such that
$X_j\tensor_{A_j} A \to X$ is an isomorphism and that for any
isomorphism of $\spec A$-curves $X \to Y$. Moreover, there is a $k \in J$ and
$\spec A_k$-curves $X_k$, $Y_k$ together with an isomorphism of $\spec
A_k$-curves $X_k \to Y_k$ such that this pulls back to the isomorphism
of $\spec A$-curves $X \to Y$. This is a somewhat technical condition
to verify, but it is very useful in the sense that it means the
resulting moduli stack is locally of finite presentation, and it
allows one to usually reduce arguments to the noetherian (even
excellent) case.{} 

The proof that $\curv_S$ is limit preserving is standard, we will
merely provide the references sufficient to prove the 
result. To obtain essential surjectivity, combine one of the
reductions used in the proof of  
\cite[Prop.\ 4.18]{MR1771927} with \cite[\textbf{IV},
4.1.4]{EGA} 
(for the dimension of fibers), \cite[\textbf{IV}, 8.10.5(xii)]{EGA}
and \cite[Thm.\ IV.3.1]{MR0302647} (for the properness),  and
\cite[\textbf{IV}, 11.2.6]{EGA} (for the flatness).  The techniques
of \cite[\textbf{IV}, 8.8.2.5]{EGA} garner full faithfulness. 
\begin{prop}\label{prop:Fppfprojective}
  Let $\pi : C \to S$ be a proper, finitely presented morphism of algebraic
  spaces. Let $s\in S$ be a closed point such that $\dim_{\kappa(s)}
  C_s \leq 1$, then there is an \'etale neighbourhood 
  $(U,\bar{u})$ of $(S,\bar{s})$ such that $C\times_S U \to U$ is
  projective. 
\end{prop}
\begin{proof}
  The statement is local on $S$ for the \'etale topology and by
  standard limit methods, we reduce immediately to the following
  situation: $S = \spec R$, where $R$ is an excellent, strictly henselian local
  ring and $s \in S$ is the unique closed point.{}

  First, assume that $C$ is a reduced scheme. Now, let $C_s\to s$
  denote the special fiber of $C\to S$. Since $C_s$ is a proper 
  scheme of dimension 1 over a field, it is manifestly
  projective. Thus, it suffices to show that the map $\Pic(C) \to \Pic(C_s)$ is
  surjective. Indeed, one can then conclude that $C$ admits a line bundle
  $\mathscr{L}$ such that the restriction to the central fiber is
  projective. By \cite[\textbf{III},
  4.7.1]{EGA}, we deduce that $\mathscr{L}$ is ample. 
  
  For this paragraph we utilize the arguments in \cite[Prop.\
  IV.4.1]{MR0463174}. Let $\mathscr{L}_s$ be a line bundle on
  $C_s$. Since $C_s \to \spec 
  \kappa(s)$ is a projective curve, to show that $\Pic(C) \to
  \Pic(C_s)$ is surjective, it suffices to treat the case
  where $\mathscr{L}_s = \Orb_{C_s}(-x_s)$, for some closed point $x_s \in
  C_s$. In an open neighborhood $U_s$ of $x_s\in C_s$, we have that
  $x_s = V(f_s) \cap U_s$, for some $f_s \in \Orb_{C_s}(U_s)$ which is
  not a $0$-divisor. Since $C_s \hookrightarrow C$ is a closed
  immersion, there is an open subscheme $U \subset C$ such that $U
  \cap C_0 = U_0$. By shrinking $U$, we may lift the equation $f_s \in
  \Orb_{C_s}(U_s)$ to $f\in \Orb_C(U)$ such that $f$ is not a
  $0$-divisor, and $V(f)\cap U \cap C_s = \{x_s\}$. In particular,
  the map $V(f)\cap U \to S$ is quasi-finite and separated. Since
  $S$ is local and strictly henselian, by \cite[\textbf{IV},
  18.12.3]{EGA}, there is a decomposition $V(f) \cap U = V_1 \amalg
  V_2$, where $V_1 \to S$ is \emph{finite} and contains
  $\pi^{-1}(s)$. Thus, by further shrinking $U$, we may assume that
  the map $V(f) \cap U \to S$ is finite. On $C$ we may now define an
  effective cartier divisor $D$ as $D\mid_{C-[V(f)\cap
    U]} = 0$ and $D\mid_U = \mathrm{div}\,f$. The cartier divisor
  $\Orb(-D)$ has the property that $\Orb_{C_s}(-D) =
  \Orb_{C_s}(-x_s)$. Since $C$ is reduced and noetherian, by
  \cite[\textbf{IV}, 21.3.4]{EGA}, $\Pic(C) \to \Pic(C_s)$ is
  surjective. 

  If $C$ is a non-reduced scheme, and $\red{C}$ is the reduction, then
  we have shown that the morphism $\red{C} \to S$ is projective. Since
  $C$ is noetherian, if $\mathscr{I}$ denotes the 
  nilradical of $C$, then there is a $k$ such that $\mathscr{I}^k
  =(0)$. Thus, it suffices to prove the following: if $i : C' \to C$ is a closed 
  immersion over $S$, defined by a square 0 ideal
  $\mathscr{J}$ such that $C'$ is projective, then $C$ is
  projective. To this end, we recall the exponential sequence on $C$:
  \[
  \xymatrix{0 \ar[r] & \mathscr{J} \ar[rr]^{a \mapsto 1 + a} & &
    \Orb_C^\times \ar[r] &  i_*\Orb_{C'}^\times \ar[r] & 1}.
  \]
  By taking the long exact sequence of cohomology, we see that the
  obstruction to lifting a line bundle on $C'$ to a line bundle on $C$
  lies in the cohomology group $H^2(C',\mathscr{J})$. Since, $C'$ is a
  projective $S$-curve, we have that $H^2(C',\mathscr{J}) = 0$. Consequently, we  
  deduce that $\Pic C \to \Pic C'$ is surjective. Hence, we may lift
  an ample bundle on $C'$ to a line bundle on $C$, and any such lift must be
  ample. 

  We now treat the case where $C$ is an algebraic space, and it
  remains to show that it is a scheme. By \cite[Thm.\ 16.6]{MR1771927}, there is a finite and surjective $S$-map  
  $\widetilde{C} \to C$, where $\widetilde{C}$ is a scheme. Since
  $\widetilde{C}$ is a proper $S$-scheme, with special 
  fiber of dimension $\leq 1$, we may conclude that $\widetilde{C}$ is
  a projective $S$-scheme. In particular, $\widetilde{C}$ has the
  Chevalley-Kleiman property (i.e. every finite set of points is contained
  in an open affine). Since $S$ is excellent,
  we may apply \cite[Cor. 48]{kollar-2008} to conclude that
  $C$ has the Chevalley-Kleiman property, thus is a scheme.
\end{proof}     
\section{Representability of The Diagonal}\label{sec:diagonal}
In this section, we will prove that the diagonal morphism $\Delta :
\curv_S \to \curv_S \times_S \curv_S$ is representable, locally of finite
presentation, separated and quasicompact. M. Artin, in \cite{MR0399094},
calls this \emph{relative representability} and as we will see, it is
an essential and natural part of the proof of algebraicity of $\curv_S$. 
Fix an $S$-scheme $T$ and let $g_1 : C_1 \to T$, $g_2 : C_2 \to T$ be
two $T$-curves. We form the 2-cartesian diagram:
\[
\xymatrix{T\times_{s_1,\curv_S,s_2} T \ar[r] \ar[d] & T \times_S T
  \ar[d]^{(s_1,s_2)}\\ \curv_S \ar[r]_{\Delta} & \curv_S \times_S \curv_S},
\]
where the $s_i$ are the induced maps to $\curv_S$ defined by the
$T$-curve $g_i$. The 2-fiber product,
$T\times_{s_1,\curv_S,s_2} T$, is isomorphic to the 
$\BETSITE{\SCH{T}}$-sheaf of isomorphisms $\Isom_T(g_1,g_2)$. That is,
the sections over a $T$-scheme $\phi : T' \to T$ are $T'$-isomorphisms
$f : \phi^*g_1 \to \phi^*g_2$. {}

To prove that $\Delta$ is representable, quasi-compact and
separated, we must show that the sheaf $\Isom_T(g_1,g_2)$ is
an algebraic space which is quasi-compact
and separated over $T$. Also, there is a $\BETSITE{\SCH{T}}$-sheaf 
$\Hom_T(g_1,g_2)$ whose sections over a morphism
$\phi :T'\to T$ are the $T'$-morphisms $f : \phi^*g_1 \to
\phi^*g_2$. One observes that $\Isom_T(g_1,g_2)$ is a subsheaf
of $\Hom_T(g_1,g_2)$.{}  

We recall the definition of the Hilbert functor for a $T$-scheme $X
\to T$: let $T' \to T$ be a morphism of schemes, let $\HILB{X/T}(T')$
be the set of isomorphism classes of closed subschemes $Z \to X\times_T
T'$ which are flat, proper, and finitely presented over $T'$. Clearly,
$\HILB{X/T} : \BETSITE{\SCH{T}} \to \SETS$ is a sheaf.{}

There is a natural transformation $\Gamma :
\Hom_T(g_1,g_2) \to \HILB{C_1\times_T C_2/T}$ which
associates to any $T'$-morphism $f : C_1\times_T T' \to C_2\times_T T'$
its graph $\Gamma_f$.
\begin{lem}\label{C:RepDiagonalStep}\label{C:RepDiagonal}
Suppose that $g_1 : C_1 \to T$, $g_2 : C_2 \to T$ are objects of
$\Coll_S$, then the $\BETSITE{\SCH{T}}$-sheaves $\Hom_T(g_1,g_2)$ and
$\Isom_{T}(g_1,g_2)$ are both representable by finitely presented and
separated algebraic $T$-spaces. 
\end{lem}
\begin{proof}
By Proposition \ref{prop:Fppfprojective}, there is an \'etale surjection 
$\phi: U \to T$ such that for $i=1$, $2$, the pullbacks, $g_{i,U} :
C_i\times_T U\to U$, are projective, flat and finitely presented. 
The inclusions $\Isom_{U}(g_{1,U},g_{2,U}) \subset
\Hom_{U}(g_{1,U},g_{2,U}) \subset \HILB{(C_{1,T} \times_{T}
  C_{2,T})\times_T U/U}$ are
representable by finitely-presented open immersions. Indeed, $\curv_S$
is limit preserving, so we may assume that $T$ is noetherian.  The
first inclusion is covered by \cite[\textbf{II}, 4.6.7(ii)]{EGA}
(without any dimension hypotheses on the fibers of $C_i$ over $T$). We
observe that the assertion for the second inclusion follows from the
first. Indeed, the latter inclusion is given by the graph homomorphism
and it has image those families of closed subschemes of $(C_{1}
\times_T C_{2})\times_T U$  for which projection onto the first
factor is an isomorphism, which as we have already seen is an open
condition.{} 

From the existence of the Hilbert scheme for finitely presented
projective morphisms, we make the following two observations:  
\begin{enumerate}
\item $\Hom_T(g_1,g_2) \times_{T} U \simeq
  \Hom_U(g_{1,U},g_{2,U})$ is representable by a separated and
  locally of finite type $S$-scheme. In particular, the morphism
  $\Hom_U(g_{1,U},g_{2,U}) \to \Hom_T(g_1,g_2)$ is \'etale and
  surjective. 
\item  The map $\Hom_U \times_{\Hom_T }  \Hom_U \rightarrow \Hom_U 
  \times_{U} \Hom_U$ is a closed immersion. Indeed, this is simply the
  locus where two separated morphisms of schemes agree. 
\end{enumerate}
Putting these together, one concludes that $\Hom_T(g_1,g_2)$ is
representable by a separated and locally of finite type algebraic
$T$-space. Since 
finitely presented open immersions are local for the \'etale topology,
we deduce the corresponding result for $\Isom_T(g_1,g_2)$.{}

All that remains is to verify that $\Hom_T(g_1,g_2)$ is quasicompact
in the case that the $g_i$ are projective. We plagiarize the argument
of \cite{dejong-2008} and include it for completeness only. Let
$\mathscr{L}_i$ be a $T$-ample line bundle for $g_i$. We may assume
that the Hilbert polynomials of the fibers of the curve $C_i \to T_i$
with respect to $\mathscr{L}_i$ are all equal to a fixed polynomial
$P_i$. Let $p_i : C_1\times_T C_2 \to C_i$ 
denote the $i$th projection and set $\mathscr{L} = p_1^*\mathscr{L}_1  
\otimes p_2^*\mathscr{L}_2$. Let $T'$ be a $T$-scheme and set $C_i' =
C_i \times_T T'$ and let $\psi_i : C_i' \to C_i$ denote the induced
map. For a $T'$-morphism $f : C_1' \to C_2'$, we have its graph 
$\Gamma_f : C_1' \to C_1'\times_{T'} C_2'$. Set 
$\psi := \psi_1 \times \psi_2: C_1'\times_{T'} C_2' \to
C_1\times_T C_2$, then we will show that the Hilbert polynomials of the
fibers of $\Gamma_f$ over $T'$ with respect to $\psi^*\mathscr{L}$ are
all equal to $P_1 +  P_2+P_1(0)$. If we show this, then we're done,
since then the map $\Hom_T(g_1,g_2) \to \HILB{C_1\times_T C_2/T}$
factors through the subfunctor of $\HILB{C_1\times_T C_2/T}$
corresponding to those flat families with Hilbert polynomial $P_1 +
P_2 + P_1(0)$.  This subfunctor is 
represented by a projective scheme, thus $\Hom_T(g_1,g_2)$ would be
quasi-projective, hence quasi-compact.{}     

It suffices to take $T' = \spec \Bbbk$, where $\Bbbk$ is a field. Let
$f : C_1' \to C_2'$ be a $T'$-morphism, then by Snapper's Theorem
\cite[Thm.\ B.7]{MR2222646}:     
\[
\chi(C_1', \Gamma_f^*\psi^*\mathscr{L}^{\tensor n}) =
\chi(C_1',\psi^*\mathscr{L}_1^{\tensor n} \tensor
f^*\psi_2^*\mathscr{L}_2^{\tensor n}) =
\chi(C_1',\psi_1^*\mathscr{L}_1^{\tensor 
  n})+ \chi(C_1',f^*\psi_2^*\mathscr{L}_2^{\tensor n}) +
\chi(C_1',\Orb_{Y_1}).  
\]
Observe that $\chi(C_1',f^*\psi_2^*\mathscr{L}_2^{\tensor n}) =
\chi(C_2',\psi^*_2\mathscr{L}_2^{\tensor n})$ and the above now expresses
that the Hilbert polynomial of the graph of $f$ is $P_1 + P_2 +
P_1(0)$. 
\end{proof}
\section{Existence of a Smooth Cover}
To construct a smooth cover of the stack $\curv_S$ by a scheme, we need to
understand the deformation theory of singular curves. A good
introduction to deformation theory is contained in \cite{MR2247603}
and \cite{MR2222646}. Our setup will be slightly different than what
appears in those sources, however. 

Throughout, we assume that $\Bbbk$ is an $S$-field, not necessarily
algebraically closed. Let $\ART{\Bbbk}$ denote the category with
objects $(A,\imath)$, where $A$ is a local artinian $S$-algebra, with
maximal ideal $\mathfrak{m}_A$, and an $S$-map $\imath : A \to
\Bbbk$. The map $\imath$ automatically induces an isomorphism of 
$S$-fields $\bar{\imath} : A/\mathfrak{m}_A \to \Bbbk$.   The 
morphisms in $\ART{\Bbbk}$ are the obvious ones. If $X$ is a
$\Bbbk$-scheme, define the functor of 
$S$-deformations $\Def_X : \ART{\Bbbk}^\OPP \to \SETS$ as follows. For
$(A,\imath)\in \ART{\Bbbk}$, $\Def_X(A,\imath)$ is the set of
isomorphism of classes of cartesian diagrams: 
\[
\xymatrix{X \ar[d] \ar[r]  &  \ar[d]
  X\tensor_{\Bbbk}{A/\mathfrak{m}_A} \ar[r] & X'
  \ar[d] \\ \spec \Bbbk \ar[r] & \spec A/\mathfrak{m}_A  \ar[r] & \spec
  A}, 
\]
where $X' \to \spec A$ is flat. Note that if we
have a morphism of deformations $X' \to X''$, then since we
necessarily have an isomorphism $X'\tensor_\Bbbk A/\mathfrak{m}_A \to
X''\tensor_{\Bbbk} A/\mathfrak{m}_A$, then $X' \to X''$ is an
isomorphism by the flatness over $A$.{} 

Let $Y$ be an $S$-scheme, if $\jmath : X \hookrightarrow Y\tensor_S
\Bbbk$ (when the context is clear, we will 
henceforth write $X \subset Y\tensor_S \Bbbk$) is a closed immersion
of $\Bbbk$-schemes, then define the embedded deformation functor 
$\Def_{X\subset Y\tensor_S \Bbbk} : \ART{\Bbbk}^\OPP \to \SETS$ as
follows. For $(A,\imath)\in 
\ART{\Bbbk}$, $\Def_{X \hookrightarrow Y\tensor_S \Bbbk}(A,\imath)$ is
the set of isomorphism of classes of cartesian diagrams: 
\[
\xymatrix{X \ar[d] \ar[r] &  \ar[d]  X\tensor_{\Bbbk}{A/\mathfrak{m}_A} \ar[r] &
  X' \ar[d] \\ Y \ar[d] \ar[r] & 
  Y\tensor_{S}{A/\mathfrak{m}_A} \ar[r] \ar[d] & Y\tensor_{S}A \ar[d] \\
 \spec \Bbbk \ar[r] & \spec  A/\mathfrak{m}_A  \ar[r] & \spec A},
\]
where $X' \to \spec A$ is flat. The same argument as before shows that
any map $X' \to X''$ of embedded deformations is an isomorphism. There
is an obvious natural transformation $\Def_{X \subset Y\tensor_S 
  \Bbbk} \to \Def_X$ given by forgetting the embedding into $Y$.
Given $(A,\imath)\in \ART{\Bbbk}$, we can define a functor $\spec
(A,\imath) : \ART{\Bbbk}^\OPP \to \SETS$ as $(A',\imath') \mapsto
\Hom_{\ART{\Bbbk}}((A,\imath),(A',\imath'))$. Note that the Yoneda Lemma
immediately implies that a map $\spec (A,\imath) \to F$, where $F$ is
a functor $F : \ART{\Bbbk}^\OPP \to \SETS$, is equivalent to an
element of $F(A,\imath)$.{}

A natural transformation of functors from $\ART{\Bbbk}^\OPP \to \SETS$, $F
\to G$ is said to be \textbf{formally smooth} if for any surjection $(A,\imath)
\to (A_0,\imath_0)$ and any diagram
\[
\xymatrix{\spec (A_0,\imath_0) \ar[d] \ar[r] & F \ar[d] \\\spec
  (A,\imath) \ar[r] \ar@{-->}[ur]^{\exists}& G}  
\]
we may fill in the dashed arrow so that it commutes. Note that if
$f : (A,\imath) \to (A_0,\imath_0)$ is a surjection, then it may be
factored into a sequence of surjections:
\[
(A,\imath) =  (A_n,\imath_n) \to (A_{n-1},\imath_{n-1})\to \cdots \to
(A_1,\imath_1) \to (A_0,\imath_0)
\]
where $\mathfrak{m}_{A_i}\ker(A_i \to A_{i-1})=0$ (this is immediate
from the Jordan-H\"older Theorem). We call such
morphisms \emph{small extensions} and note that any such morphism has
square 0 kernel.{}

The next two results are to be considered folklore in this generality.
For similar statements, with stronger hypotheses, see for example
\cite[Prop.\ 3.2.9]{MR2247603} and \cite[Cor.\ 8.5.32]{MR2222646}. 
\begin{thm}\label{thm:main_defthy}
  Suppose that $X$ is a projective $\Bbbk$-scheme, with
  $h^2(\Orb_X)=0$. Consider an embedding $X \hookrightarrow \Pr^N$
  such that $h^1(X,\Orb_X(1)) = 0$, then $\Def_{X \subset \Pr^N} \to
  \Def_X$ is formally smooth.
\end{thm}
We will prove this in a moment. The following is a variant of
\cite[Thm.\ 8.5.31]{MR2222646}, with a supplied proof.
\begin{prop}\label{prop:illusie}
  Let $X$ be a proper $\Bbbk$-scheme and consider an
   embedding $\jmath : X 
  \hookrightarrow Y\tensor_S \Bbbk$, where $Y$ is a smooth $S$-scheme,
  then if $H^1(X,\jmath^*T_{Y\tensor_S\Bbbk/\Bbbk}) = 0$, $\Def_{X
    \subset Y\tensor_S \Bbbk} \to \Def_X$ is formally smooth. 
\end{prop}
\begin{proof}
  Fix a small extension $(A_1,\imath) \to (A_0,\imath_0)$ and let $K =
  A_1/\mathfrak{m}_{A_1} = A/\mathfrak{m}_{A_0}$. Consider a diagram 
  \[
  \xymatrix{X \ar[r] \ar[d] & \ar[d] Y\tensor_S \Bbbk \\
    X\tensor_\Bbbk K \ar[r] \ar[d] & Y_K := Y\tensor_S K \ar[d]\\
    X_{0} \ar@{^{(}->}[r]  \ar[d]  & Y_{A_0}:=Y\tensor_S A_0 \ar[d] \\
    X_{1} \ar@{-->}[r] & Y_{A_1}:=Y\tensor_S A\\
  }
  \]
  where $[X_1]  \in \Def_{X}(A_1)$, $[X_0 \hookrightarrow
  Y_{A_0}] \in \Def_{X \subset Y\tensor_S A_0}(A_0)$, and each
  restrict to $[X_0] \in \Def_{X}(A_0)$. To show that $\Def_{X \subset
    Y\tensor_S \Bbbk} \to \Def_X$ is formally smooth, it suffices to
  construct a map $X_{1}   \rightarrow Y_{A_1}$, since any such map is 
  automatically a closed immersion. Indeed, the morphism is affine by
  using Serre's Criterion, and by the Nakayama Lemma for modules over an
  artinian ring, it is a closed immersion, because it is a closed
  immersion modulo a nilpotent ideal.
  
  For $i=0$, $1$, let $S_i = \spec A_i$, and consider the composition
  of morphisms $X_0 \xrightarrow{f} Y_{A_1} \to S_1$. By
  \cite[II.2.1.2]{MR0491680} there is a distinguished triangle: 
  \[
  \xymatrix{f^*L_{Y_{A_1}/S_1} \ar[r] &  L_{X_0/S_1} \ar[r] &  L_{X_0/Y_{A_1}}}.
  \]
  Note that since the closed immersion $S_0 \to S_1$ is defined by a
  square 0 sheaf of ideals $I$, it is supported on $\spec K$ and hence
  $S_0$. Let $s_0 : X_0 \to S_0$ be the structure map, taking the long
  exact sequence associated to $\Hom_{X_0}(-,s_0^*I)$, gives an exact
  sequence:  
  \[
  \xymatrix{\Ext^1_{X_0}(L_{X_0/Y_{A_1}}, s_0^*I)\ar[r]  & \ar[r]
    \Ext^1_{X_0}(L_{X_0/S_1},s_0^*I) &
    \Ext^1_{X_0}(f^*L_{Y_{A_1}/S_1},s_0^*I)}.  
  \]
  Since $Y_{A_1} \to S_1$ is smooth, then $L_{Y_{A_1}/S_1} \cong
  \Omega_{Y_{A_1}/S_1}$, by \cite[III.3]{MR0491680}. In particular, if
  $\jmath_0 : X_0 \to Y_{A_0}$ is the inclusion and $\sigma : Y_{A_0}
  \to Y_{A_1}$ is the induced map from $S_0 \to S_1$, then
  $\sigma\jmath_0 = f$ and 
  \[
  f^*L_{Y_{A_1}/S_1} \cong f^*\Omega_{Y_{A_1}/S_1} \cong
  \jmath_0^*\sigma^*\Omega_{Y_{A_1}/S_1} \cong \jmath_0^*\Omega_{Y_{A_0}/S_0},
  \]
  since differentials pullback along the base. By
  \cite[$\mathbf{0}_{\textrm{III}}$, 12.3.5]{EGA} and
  \cite[$\mathbf{0}_{\textrm{I}}$, 5.4]{EGA}, since 
  $\jmath_0^*\Omega_{Y_{A_0}/S_0}$ is locally free of finite rank and
  $s_0^*I$ is coherent, then   
  \begin{align*}
    \Ext^1_{X_0}(f^*L_{Y_{A_1}/S_1},s_0^*I) &\cong
    H^1(X_0,Hom_{\Orb_{X_0}}(\jmath_0^*\Omega_{Y_{A_0}/S_0},s_0^*I))\\
    &\cong
    H^1(X_0,Hom_{\Orb_{X_0}}(\jmath_0^*\Omega_{Y_{A_0}/S_0},\Orb_{X_0})
    \tensor_{\Orb_{X_0}} s_0^*I)\\
    &\cong H^1(X_0,\jmath_0^*T_{Y_{A_0}/S_0}\tensor_{\Orb_{X_0}}
    s_0^*I)\\
    &\cong
    H^1(X_0,\jmath_0^*T_{Y_{A_0}/S_0}\tensor_{s_0^{-1}\Orb_{S_0}} I)\\
    &\cong H^1(X_0,\jmath_0^*T_{Y_{A_0}/S_0})\tensor_{S_0} I,
  \end{align*}
  because $s_0 : X_0 \to S_0$ is flat. Noting that the coherent sheaf
  $\jmath_0^*Y_{Y_{A_0}/S_0}$ is flat over the artinian local scheme
  $S_0$ and $H^1(X_K,\jmath_K^*T_{Y_K/K}) =
  H^1(X,\jmath^*T_{Y_\Bbbk/\Bbbk})\tensor_\Bbbk K = 0$, then
  \cite[Exercise III.11.8]{MR0463157} implies that
  $H^1(X_0,\jmath_0^*T_{Y_{A_0}/A_0}) = 0$. Thus,
  $\Ext_{X_0}(f^*L_{Y_{A_1}/S_1},s_0^*I)=0$.{} 

  We now apply \cite[III.1.2]{MR0491680} to observe that our original 
  exact sequence (together with the vanishing result proved above)
  provides a surjection: 
  \[
  \xymatrix{\Exal_{\Orb_{Y_{A_1}}}(\Orb_{X_0},s_0^*I) \ar[r] &
    \Exal_{\Orb_{S_1}}(\Orb_{X_0},s_0^*I) \ar[r] &  0 }.  
\]
In particular, $[X_{0}]$ is an element of
$\Exal_{\Orb_{S_1}}(\Orb_{X_0},s_0^*I)$ and so there is an
$\Orb_{Y_{A_1}}$-extension of $\Orb_{X_0}$ by $s_0^*I$ corresponding to
$X_{A_1}$ (indeed, it is given by the sheaf of algebras
$\Orb_{X_{A_1}}$). Hence, there is a map of sheaves of algebras $\Orb_{Y_{A_1}}
\to \Orb_{X_{1}}$ and consequently a morphism of schemes $X_1 \to Y_{A_1}$,
which extends $X_0 \to Y_{A_1}$. 
\end{proof}
\begin{proof}[Proof of Theorem \ref{thm:main_defthy}]
  Since the Euler sequence is exact, we may pull it back and dualize
  it to obtain an exact sequence:
  \[
  \xymatrix{0 \ar[r] & \Orb_X \ar[r] & \Orb_X(1)^{\oplus (N+1)} \ar[r]
    & \imath^*T_{\Pr^N} \ar[r] & 0}.
  \]
  Taking the long exact sequence of cohomology, we arrive at the
  following segment:
  \[
  \xymatrix{H^1(\Orb_X(1))^{N+1} \ar[r] & H^1(X,\imath^*T_{\Pr^N})
    \ar[r] & H^2(\Orb_X)}.
  \]
  The assumptions ensure that $H^1(X,\imath^*T_{\Pr^N}) = 0$. An
  application of Proposition \ref{prop:illusie} proves the result.
\end{proof}
With the relevant deformation theory in place, we can now complete the
proof of Theorem \ref{thm:stack_of_curv}.
\begin{proof}[Proof of Theorem \ref{thm:stack_of_curv}]
  We have shown that $\curv_S$ is a limit
  preserving stack over $\BETSITE{\SCH{S}}$. In \S\ref{sec:diagonal},
  we proved that the diagonal is represented by finitely presented,
  quasi-compact, and separated algebraic $S$-spaces. Thus, to show
  that $\curv_S$ is an algebraic $S$-stack, it remains to construct an
  $S$-scheme
  $U$ together with a smooth, surjective $S$-morphism $U \to \curv_S$.{}

  To prove the existence of a smooth cover, we consider a geometric
  point $\spec \Bbbk \to \curv_S$. This corresponds to $C \to \spec \Bbbk$,
  for some 1-dimensional projective scheme $C$. Pick a $\Bbbk$-embedding
  $C \to \Pr^M_{\Bbbk}$ such that $H^1(\Orb_C(1)) = 0$. Let $U_{C/\Bbbk}$
  denote an affine open neighbourhood of the point $x: \spec \Bbbk \to
  \mathrm{Hilb}_{\Pr^M_S/S}$ and take $V_{C/\Bbbk} \to U_{C/\Bbbk}$ to denote
  the universal family. By Cohomology and Base Change
  \cite[Thm.\ 12.11]{MR0463157}, we may replace $U_{C/\Bbbk}$ (and so
  $V_{C/\Bbbk}$ changes also) by an affine open subset containing the
  image of $x$ such that $h^1(V_{C/\Bbbk,v},\Orb_{V_{C/\Bbbk,{v}}}(1)) = 0$
  and all fibers are flat of
  dimension 1 for all points ${v} \to U_{C/\Bbbk}$. {}

  In particular, we obtain a finitely presented morphism $U_{C/\Bbbk} \to
  \curv_S$. We will now proceed to show that this morphism is in fact
  smooth. That is, if $T \to \curv_S$ is an $S$-morphism, then
  $p_T : T' := T\times_{\curv_S} U_{C/\Bbbk} \to T$ is smooth. Since
  $\curv_S$ is limit preserving, we may assume that $T$ is
  noetherian. By taking a faithfully flat \'etale cover of 
  $T$, by Proposition \ref{prop:Fppfprojective}, we may assume that the family
  of curves corresponding to the map $T \to \curv_S$ is
  projective. Thus, the map $p_T$ is a map of schemes.  To show that
  the map $p_T$ is smooth, 
  by \cite[\textbf{IV}, 17.14.2]{EGA}, it suffices to fix $t' \in
  T'$, let $t=p_T(t')\in T$ and we need to show that there is an arrow
  completing the diagram: 
  \[
  \xymatrix{\spec A/I \ar[r] \ar[d] & T'
    \ar[d] \ar[r] & U_{C/\Bbbk} \ar[d] \\ \spec A \ar[r] \ar@{-->}[ur]& T \ar[r]
    & \curv_S }
  \]
  where $A$ is an artin local ring, $I\ideal A$ is an ideal, with
  the closed point of $\spec A$ mapping to $t$, and the residue field
  of $A/I$ is the same as the residue field of $t'$, $K$. Let
  the $\spec A$-curve $C_A \to \spec A$ denote that induced by the
  morphism $\spec A \to \curv_S$. Then the $2$-commutativity of the
  diagram implies that there is a $\spec A/I$-embedding $C_A\tensor
  (A/I) \to \Pr^M_{A/I}$. Let $C_K \to \spec K$ denote the
  $K$-curve $C_A \tensor_A K$. We have a map $u : \spec
  K \to U_{C/\Bbbk}$, and a $K$-isomorphism $V_{C/\Bbbk,u} \to
  C_K$, thus $h^1(C_K, \Orb_{C_K}(1)) = 0$.  By Theorem
  \ref{thm:main_defthy}, the map $\Def_{C_K \subset \Pr^M_K} \to
  \Def_{C_K}$ is formally smooth, hence a map $\spec A \to
  U_{C/\Bbbk}$ exists completing the given diagram. 
 
  The isomorphism classes of morphisms $Y \to \curv$, where $Y$ is an
  affine open subscheme of $H_N$ for some $N$, form a set. Take $U$ to
  be the disjoint union over all those such $Y \to \curv$ which are
  smooth. By the above, $U \to \curv$ is smooth and surjective.
\end{proof}
\section{Proofs of the Corollaries}\label{sec:cors}
In this section, we run through the proofs of the corollaries. 
\begin{proof}[Proof of Corollary \ref{cor:stack_of_curv1}]  
  First, we show that the universal curve $\curv_{S,1}$ is an algebraic
  stack, which is locally of finite type over $S$. This is obvious:
  one has the forgetful morphism $\curv_{S,1} \to \curv_S$ (given by
  forgetting the section of the family) and this morphism is
  representable. Indeed, for an affine scheme $T \to \curv_S$
    (corresponding to a $T$-curve), the 2-fiber product
    $T \times_{\curv_S} \curv_{S,1} \cong \Hom_T(T,C)=C$. Noting that  
  \[
  \curv_{S,n} \cong \underbrace{\curv_{S,1} \times_{\curv_S} \cdots
    \times_{\curv_S} \curv_{S,1}}_n,
  \]
  we conclude the general case. 
\end{proof}
\begin{proof}[Proof of Corollary \ref{cor:stack_of_curv2}] {\par\indent}
  In all cases, we need to check that if $C \to T$ is a $T$-curve,
  then the locus in $T$ which satisfies the condition has a natural
  scheme structure. Cases (1)-(9) are all open conditions, by the
  results of \cite[\textbf{IV}, \S 12.2]{EGA}. {}

 For the remainder of the cases, we may reduce to the noetherian case
  as follows: we will reduce to the case of a noetherian base and
  a projective family.  Since $\curv$ is limit preserving, $f$
  factors as $\spec A \to \spec A_0 \to \curv$, where $A_0$ is of
  finite type over $\mathbb{Z}$. All the conditions are geometrically
  fibral, so we may work over a faithfully flat \'etale extension of
  the base $\spec A_0$. Hence, it suffices to consider those families
  of curves which are projective over an affine noetherian base.
  For the remainder, we fix an object $f : C\to T$ of $\curv_S$, where
  $T$ is the spectrum of the noetherian ring $A$ and we assume that
  $f$ is projective.
  \begin{enumerate}
  \item[(10)] \cite[Cor.\ III.9.3]{MR0463157} shows this condition
    is open and closed.
  \item[(11)] Follows immediately from Chevalley's semicontinuity
    Theorem. 
  \item[(12)] It suffices to show that if $C \to S$ is an object of 
    $\curv_S$, where $S$ is the spectrum of a noetherian ring $A$,
    then $S \times_{\curv} S \to S$ being unramified is an open
    condition on $S$. This will follow from the more
    general assertion: let $p: G \to S$ be a locally of
    finite type, group algebraic space, with $S$ noetherian, then if
    $\bar{s} \to S$ is a geometric point and the group scheme
    $G_{\bar{s}} \to \bar{s}$ is unramified, then there is an open
    subscheme $U$ of $\bar{s}$ such that $G_U \to U$ is
    unramified. Observe that we can find an open subspace $W$
    containing $G_{\bar{s}} \subset 
    G$ such that $p\mid_W : W \to S$ is unramified. Let $e : S \to G$
    be the identity section and $\imath : W \to G$ the immersion, then
    the fiber product $U = W \times_G S$ is an open subscheme of
    $S$. Moreover, for any geometric point $\bar{u} \to U$ we have
    $G_{\bar{u}} \to \bar{u}$ is unramified on an open subscheme of
    the identity and by using translations in this group, we can cover
    it by unramified open subschemes.
  \item[(13)] If the complex has flat cohomology over the base, then it is
    immediate from cohomology and base change. If the cohomology is
    not flat, take a flattening stratification (the morphism is
    projective, so these exist), then apply the earlier case. In this
    situation, you obtain a locally closed substack (as opposed to an
    open substack).
  \end{enumerate}
\end{proof}
To prove Corollary \ref{cor:stack_of_curv3}, it remains to show that
the stack is quasicompact. We proceed to prove the relevant
boundedness results. The following argument is due to F. Van Der
Wyck.
\begin{lem}\label{lem:bd_embd}
  Let $\Bbbk$ be an algebraically closed field and $S$ a reduced
  curve singularity, then $S$ may be embedded in an affine space of 
  dimension $\leq (\delta_S+1)^2$.
\end{lem}
\begin{proof}
  First, suppose that $S$ is unibranched, then it is a finitely
  generated subalgebra of $\Bbbk[[t]]$. Let $f_1$, $\dots$, $f_r$
  denote a set of generators and we may assume that the
  degree of each $f_i$ is distinct. Let $M$ denote the semigroup
  generated by the degrees of the $f_i$. Observe that if $n=\min_i
  \{\deg f_i\}$, then $n \in M$. In particular, it follows that there
  are at most  $n-2$ other generators (by inspection of the residues),
  since the $\deg f_i$ are all distinct. Note that $n \leq \delta_S$
  and so the embedding dimension for a unibranched singularity is
  $\leq \delta_S + 1$.{}

  Now suppose there are $r$ branches, then $S \subset \prod_{i=1}^r
  \Bbbk[[t_i]]$. Note that the $\delta$ of a branch is bounded by
  $\delta_S$ and the number of branches is bounded by $\delta_S +
  1$. The former is obvious, the latter clear from the observation
  that $S$ doesn't contain the elements $t_i$ or $(t_1,\dots,t_r)$ and
  there are $r+1$ of these. Hence, the embedding dimension is bounded
  by $(\delta_S+1)^2$. 
\end{proof}
The following argument had inputs from D. Smyth and R. Vakil.
\begin{thm}\label{thm:bd_deg}
  Suppose that $C$ is a connected curve with $\leq e$
  irreducible components with arithmetic genus $g$, then there is an
  embedding $C \hookrightarrow \Pr^{N_{g,e}}$ such that $\deg C \leq 
  D_{g,e}$, where 
  \begin{align*}
    N_{g,e} &= (g+e)^2 + 1 \\
    D_{g,e} &= 2e(g+e-1)(g+e) + e^2. 
  \end{align*}
\end{thm}
\begin{proof}
  We first determine $D_{g,e}$. Let $C_{\mathrm{sm}} \subset C$
  denote the smooth locus, then $C_{\mathrm{sm}}$ is a disjoint union
  $\amalg_{i=1}^e W_i$. For each $i=1$, $\dots$, $e$, take $p_i
  \in W_i$. Let $Z$ be the divisor $p_1 + \cdots + p_e$ and let
  $\mathscr{L} = \Orb(Z)$, then $\deg \Orb(Z) = e$. It suffices to
  find some $m=m(g,e)$ (depending only on $g$, $e$) such that
  $\mathscr{L}^{m}$ is very 
  ample. Indeed, we would then have $D_{g,e} = me$. We need to show that
  $\mathscr{L}^{m}$ separates points and tangent vectors. Thus, it
  remains to show that for any $c\in C$:
  \[
  H^1(C,\mathscr{L}^{m}(-c)) = H^1(C,\mathscr{L}^{m}(-2c)) = 0.
  \]
  Using the standard exact sequence relating these two ideal sheaves,
  the vanishing of the former is determined by the vanishing of the
  latter.{}

  Let $\pi : \widetilde{C} \to C$ be the
  normalization map and let $c \in C$. Take $c_1$, $\dots$, $c_{r_c}$ to
  denote the points of the fiber $\pi^{-1}(c)$ and let $\delta_c$
  be the $\delta$-invariant of $c$, then there is an exact
  sequence 
  \[
  \xymatrix{0 \ar[r] &
    \pi_*\Orb_{\widetilde{C}}(-2\delta_c(c_1+\cdots+c_{r_c})) \ar[r] &
    \Orb(-2c) \ar[r] & \mathscr{E} \ar[r] & 0},
  \]
  with $\mathscr{E}$ supported only on $c$. Twisting this exact
  sequence by $\mathscr{L}^m$ (for some $m$ yet to be determined) and
  taking the exact sequence of cohomology we obtain:
  \[
  \xymatrix{H^1(C,\mathscr{L}^m\tensor\pi_*
    \Orb_{\widetilde{C}}(-2\delta_c(c_1+\cdots+c_{r_c}))) \ar[r] &
    H^1(C,\mathscr{L}^m(-2c)) \ar[r] & 0
  }. 
  \]
  Since $\pi_*$ is exact, we obtain from the projection formula:
  \[
  H^1(C,\mathscr{L}^m\tensor\pi_*
    \Orb_{\widetilde{C}}(-2\delta_c(c_1+\cdots+c_{r_c}))) =
    H^1(\widetilde{C},
    (\pi^*\mathscr{L})^m(-2\delta_c(c_1+\cdots+c_{r_c}))). 
  \]
  Taking $m_c = 2r_c\delta_c + e$ and applying \cite[Cor.\ IV.3.3
  and Exercise III.7.1]{MR0463157} furnishes us with:
  \[
  H^1(\widetilde{C},(\pi^*\mathscr{L})^{m_c}(-2\delta_c
  (c_1+\cdots+c_{r_c})))  = 0.
  \]
  We may conclude that $H^1(C,\mathscr{L}^{m_c}(-2c)) = 0$.{}
  
  Note that $g = p_a(\widetilde{C}) + \sum_{c \in C} \delta_c$. In
  particular, since the number of connected components of
  $\widetilde{C}$ is bounded by the number of irreducible components
  of $C$, we have
  \[
  \sum_{c \in C} \delta_c \leq   g + e-1.
  \]
  Hence, the $\delta$s of the singular points of the $C$ are bounded
  by $g+e-1$   and by the proof of Lemma \ref{lem:bd_embd}, we see that
  the number of branches over each singular point is bounded by
  $g+e$. Hence, if 
  we take $m = 2(g+e-1)(g+e) + e$, then we are done and we have
  bounded the degree. Continuing with these ideas, an application of
  Lemma \ref{lem:bd_embd} implies that the embedding dimension of
  every singularity of $C$ is bounded by $(g+e)^2$.{} 

  Using our line bundle $\mathscr{L}^m$, we produce an embedding $C 
  \hookrightarrow \Pr^M_\Bbbk$ and let 
  $\mathrm{Sec}\,(C)$ be the secant variety of $C$, this has
  dimension bounded by 3. Take $\mathrm{Tan}\,(C)$ to denote the
  tangent variety, then this has dimension bounded by $(g+e)^2 + 1$
  by the above bound on embedding dimension.  By choosing a point $P$
  not in $\mathrm{Sec}\,(C) \cup \mathrm{Tan}\,(C)$ and projecting
  from it, we obtain $N_{g,e} \leq (g+e)^2+1$. Note that the degree of
  the embedding $C \hookrightarrow \Pr^{N_{g,e}}$ remains $\leq
  D_{g,e}$.
\end{proof}
\begin{proof}[Proof of Corollary \ref{cor:stack_of_curv3}]
  From the proof of Corollary \ref{cor:stack_of_curv1}, it suffices to
  show that $\mathfrak{W}_{g,e,0}$ is quasicompact. Let $W_{g,e}$ 
  denote the Hilbert scheme of curves in $\Pr^{N_{g,e}}$, whose fibers
  are embedded curves of arithmetic genus $g$, with less than $e$
  irreducible components and of degree $\leq D_{g,e}$. Note that
  $W_{g,e}$ is quasicompact, since the component of the Hilbert scheme
oh   corresponding to a fixed Hilbert polynomial is projective and the
  Hilbert polynomial of a curve is completely determined by its degree
  and genus.{} 

  Let $\mathscr{W}_{g,e}\to W_{g,e}$ denote the universal family, then
  we have an induced morphism $W_{g,e} \to \mathfrak{W}_{g,e,0}$. By
  Theorem \ref{thm:bd_deg}, this map is surjective. We conclude that
  $\mathfrak{W}_{g,e,0}$ is quasicompact.
\end{proof}
\bibliography{/Users/me/Bibliography/references}
\bibliographystyle{/Users/me/Bibliography/dary}

\end{document}